\documentclass[12pt,a4paper]{article}
\usepackage{mathtools}
\usepackage{amsthm}
\usepackage{amsfonts}
\usepackage{hyperref}

\usepackage[T1]{fontenc}
\usepackage[sc]{mathpazo}
\numberwithin{equation}{section}

\newtheorem*{theorem}{Theorem}
\newtheorem{lemma}{Lemma}
\newtheorem{corollary}[lemma]{Corollary}
\theoremstyle{definition}

\DeclareMathOperator{\diag}{diag}

\def\expr#1{{\mathsf{#1}}}

\def\bigO{{\mathrm{O}}}

\def\K{{\mathbf{K}}}
\renewcommand\d{{\mathrm{d}}}
\def\conj#1{\overline{#1}}

\def\C{{\mathbb{C}}}

\renewcommand\P{{\mathbb{P}}}

\def\ZZ{Z,\conj{Z}}

\def\xy{x,\conj{y}}
\def\xx{x,\conj{x}}
\def\yy{y,\conj{y}}
\def\yx{y,\conj{x}}

\def\zz{z,\conj{z}}

\def\zo{z,\conj{0}}

\def\defas{:=}

\def\gt{>}
\def\lt{<}

\def\approaches{\rightarrow}

\def\Modulus#1{\left|#1\right|}

\def\norm#1{\| #1\|}

\def\diastasis{\tilde{d}}
\def\AffineBall#1#2{\{y:\, \diastasis(#1,y) < #2\}}
\def\EuBall#1{B(#1)}

\def\LocExpOfQ{h}

\def\twistSheaf#1{{\mathcal{O}(#1)}}
\def\BoundForLocDi{C_p}
\def\PosReals{[0,\infty)}
\def\BoundForGloDi{M_p}
\def\PosInt{\mathbb{N}}

\begin{document}

\title{Eventual positivity of Hermitian polynomials and integral operators
            \footnote{Article accepted by {\textit{Chin. Ann. Math. Ser. B}}}}
\author{Colin Tan \thanks{Email address: colinwytan@gmail.com}}
\date{}
\maketitle

\begin{abstract}
Quillen proved that, if a Hermitian bihomogeneous polynomial is strictly positive on the unit sphere,
    then repeated multiplication of the standard sesquilinear form to this polynomial eventually results in a sum of Hermitian squares.
Catlin-D'Angelo and Varolin deduced this positivstellensatz of Quillen from the eventual positive-definiteness of an associated integral operator.
Their arguments involve asymptotic expansions of the Bergman kernel.
The goal of this article is to give an elementary proof of the positive-definiteness of this integral operator.
\end{abstract}



\bigskip A central problem in real geometry is to establish certificates that directly witness the positivity of an algebraic morphism.
The first of such certificates, known as positivstellensatze, was Artin's 1927 solution to Hilbert's 17th problem \cite{Artin}.
After a hiatus of four decades, Quillen \cite{Quillen68} proved the first Hermitian postivstellensatz, thereby filling a gap in the literature.
His positivestellensatz states that, if a Hermitian bihomogeneous polynomial is strictly positive on the unit sphere,
    then repeated multiplication of the standard sesquilinear form to this polynomial eventually results in a sum of Hermitian squares.
We will deduce this as Corollary \ref{cor: QuillenPositivstellensatz} from the main result of this article.

Quillen's positivstellensatz has attracted several proofs \cite{CatlinDAngelo97, CatlinDAngelo99, DrouotZworski13, PutinarScheiderer10, ToYeung06}.
Some of these approaches lead to further improvements of Quillen's result.
For instance, Catlin-D'Angelo \cite{CatlinDAngelo99} gave a generalized embedding theorem of holomorphic vector bundles.
To-Yeung's positivstellensatz \cite{ToYeung06} is a more precise refinement of Quillen's result.
Putinar-Scheiderer \cite{PutinarScheiderer10} gave pesudoconvex bounaries other than the unit sphere on which every strictly
    positive algebraic morphism is a sum of Hermitian squares.

In 1997, Catlin-D'Angelo \cite{CatlinDAngelo97} independently rediscovered Quillen's result.
They observed that this positivstellensatz of Quillen is equivalent to the eventual positive-definiteness of an associated integral operator.
They then showed that this integral operator is well approximated by the Bergman kernel in the limit.
Later, this approach was taken by Varolin \cite{Varolin08}.
In this article, we give an elementary proof that follows Varolin's approach and deduce Quillen's positivstellensatz.

The global holomorphic sections of the tautological line bundle $\twistSheaf{1} \to \P^n$ over complex projective space form a complex vector space $H^0(\P^n, \twistSheaf{1})$.
Fix a basis $(\Phi_0,\ldots, \Phi_n)$ of $H^0(\P^n, \twistSheaf{1})$.
Let $r \defas |\Phi_0|^2 + \cdots + |\Phi_n|^2$.
Then $r$ induces a Hermitian metric $(s,\conj{s}) \mapsto |s|^2/r$ on $\twistSheaf{1}$ whose curvature is a Fubini-Study K\"{a}hler form on $\P^n$.
More generally, given a nonnegative integer $d$, a Hermitian metric $p$ on $\twistSheaf{d}\to P^n$ is {\emph{globalizable}}
    if there exists a family $\{a_{\alpha\beta}\}_{|\alpha|=|\beta|=d}$ of complex constants doubly indexed by multiindices $\alpha$ and $\beta$ of length $d$ such that
\[
p(s,\conj{s}) = \frac{|s|^2}{\sum_{|\alpha| = |\beta| = d} a_{\alpha\beta} \Phi^\alpha \conj{\Phi}^\beta}
\]
Here $\Phi^\alpha \conj{\Phi}^\beta \defas \Phi_0^{\alpha_0} \Phi_1^{\alpha_1} \cdots \Phi_n^{\alpha_n}\conj{\Phi_0}^{\beta_0} \conj{\Phi_1}^{\beta_1}\cdots \conj{\Phi_n}^{\beta_n}$.
Catlin-D'Angelo introduced this concept of gloablizability of Hermtitian metric in \cite{CatlinDAngelo99}.

Normalize the Fubini-Study volume form $\Omega$ such that $\int_{\P^n} \Omega=1$.
Let $m$ be a nonnegative integer.
Equip the complex vector space $H^0(\P^n, \twistSheaf{m + d})$ of global holomorphic sections of $\twistSheaf{m + d}$ with an inner product
\begin{equation} \label{eq: DefineInnerProductOnVectorSpaceOfHolomorphicPolynomials}
(s_1,s_2) \defas \int_{\P^n} \frac{s_1\conj{s_2}}{r^mp} \, \Omega
\end{equation}
The induced norm is given by $\norm{s} =\sqrt{(s,s)}$.
Associate to $r^m p$ a sesquilinear form $\K_{r^mp} : H^0(\P^n, \twistSheaf{m + d})\times H^0(\P^n, \twistSheaf{m + d}) \to \C$ given by
\begin{equation} \label{eq: DefineAssociatedIntegralOperator}
\K_{r^mp}(s_1, s_2) = \int_{\P^n} \! \int_{\P^n} \! \frac{(r^mp)(\xy)\conj{s_2(x)} s_1(y)}{(r^mp)(\xx)(r^mp)(\yy)} \, \Omega(y) \Omega(X)
\end{equation}
In this article, we show that $\K_{r^mp}$ is eventually positive definite.
\begin{theorem}
Let $d$ be a nonnegative integer and $p$ be a Hermitian metric on $\twistSheaf{d} \to \P^n$.
If $p$ is globalizable, then for $m$ sufficiently large, the following asymptotic holds uniformly for $s\in H^0(\P^n, \twistSheaf{m + d})$:
\begin{equation} \label{eq: ComputeAsymptoticsOfEventuallyPositiveDefiniteKernel}
\K_{r^mp}(s,s) = \left\{\frac{n!}{m^n}+ \bigO\left(\frac{(\log m)^{n+2}}{m^{n+1}}\right)\right\} \norm{s}^2
\end{equation}
\end{theorem}
The author is aware that
    the above result (in fact a stronger asymptotic without the $(\log m)^{n+2}$ factor) would follow from pp. 313-314 of \cite{Varolin08},
    but is unable to follow the argument provided there.

In our proof, we show that the double integral \eqref{eq: TheDoubleIntegral} which represents this integral operator concentrates in a tubular neighbourhood of the diagonal with radius $(\log m)/\sqrt{m}$.
This concentration result is inspired by the asymptotically concentration of the Bergman kernel along the diagonal.
Our choice of $(\log m)/\sqrt{m}$ as radius is influenced by Tian,
    who used the same radius to construct peak sections in \cite{Tian90} to prove the convergence of Bergman metrics.

\section{Some lemmas} \label{sec: LemmasToProveAsymptotic}

We use the following notation and conventions throughout this article.
Unless stated otherwise, asymptotics in this article are taken in an integer $m$ that approaches infinity.
Following Knuth \cite{Knuth92}, the {\emph{Iverson bracket}} of a proposition $\tau$ is the quantity
\[
[\tau] \defas \begin{cases}
1 &{\text{if $\tau$ is true}} \\
0 &{\text{if $\tau$ is false}}
\end{cases}
\]
    For example, the characteristic function of a subset $E$ of $\P^n$ is given by $\chi_E(y)= [y\in E]$.
    Another example is the Kronecker delta, which is given by $\delta_{ij}=[i=j]$.

Recall from the introduction that $r \defas |\Phi_0|^2 + \cdots +|\Phi_n|^2$ for some chosen basis $(\Phi_0, \ldots, \Phi_n)$ of $H^0(\P^n, \twistSheaf{1})$.
This globalizable metric $r$ can be polarized to yield a metric on $\P^n$.
This metric $\diastasis: \P^n\times \P^n \to [0,\infty]$ is given by
\begin{align*}
\diastasis(x, y) &= \left(\frac{r(\xx)r(\yy)}{|r(\xy)|^2} - 1\right)^{1/2} \\
    &= \left(\frac{(|\Phi_0(x)|^2 + \cdots + |\Phi_n(x)|^2) (|\Phi_0(y)|^2 + \cdots + |\Phi_n(y)|^2)}
              {|\Phi_0(x)\conj{\Phi_0(y)} + \cdots + \Phi_n(x)\conj{\Phi_n(y)}|^2} - 1\right)^{1/2}
\end{align*}
For each point $x$ of $\P^n$, there exists a {\emph{canonical coordinate}} $z$ centered at $x$ such that
\begin{align*}
\diastasis([1:0], [1:z])
&=  \left(\frac{(1^2 + 0^2 + \cdots + 0^2)(1^2 + |z_1|^2 + \cdots + |z_n|^2)}
               {|1\cdot \conj{1} + 0\cdot \conj{z_1} + \cdots + 0\cdot \conj{z_n}|^2} - 1\right)^{1/2}\\
&= |z|
\end{align*}
For example, if $x$ is a point of $\P^n$, then the subset $\AffineBall{x}{\infty}$ is biholomorphic to $\C^n$.

\begin{lemma} \label{lem: BoundLocallyDiastasisOfP}
Let $p$ be a globalizable metric on $\twistSheaf{d}\to \P^n$.
There exists a positive constant $\BoundForLocDi$ such that
\[
\Modulus{
                \left[\frac{|p(\xy)|^2}{p(\xx)p(\yy)}\right]^{1/2}-\left[\frac{p(\xx)p(\yy)}{|p(\xy)|^2}\right]^{1/2}
            } \le  \BoundForLocDi \diastasis(x,y)^2
\]
\end{lemma}

\begin{proof}
Fix a point $x$ of $\P^n$.
Define a function $G : \P^n \to \PosReals$ by
\begin{equation} \label{eq: definition of G}
G(y) = \left[\frac{|p(\xy)|^2}{p(\xx)p(\yy)}\right]^{1/2}
\end{equation}
Choose a trivialization of $\twistSheaf{d}$.
Choose a canonical coordinate $z$ centered at $x$.
In this trivialization and coordinate,
\[
2\log G(z) = \log p(z,\conj{0}) + \log p(0,\conj{z}) - \log p(z,\conj{z}) - \log p(0,\conj{0})
\]
Taking the holomorphic derivative,
\begin{align*}
2\frac{\partial G(z)}{G(z)}
&= \frac{\partial p(\zo) }{p(\zo)} + 0 - \frac{\partial p(\zz) }{p(\zz)} - 0
\end{align*}
Noting that $G(0) = 1$, evaluation at $z=0$ gives $\partial G(0) = 0$.
The chain rule  $\partial (G^{-1}) = - \partial G / G^2$ implies $\partial (G^{-1})(0) = 0$.
Since $\conj{G} = G$, we also have the vanishing of the antiholomophic derivatives, namely $\conj{\partial} G(0) = \conj{\partial}(G^{-1})(0) = 0$.

Noting that $G(0)-(G^{-1})(0) = 0$, the Taylor theorem gives local functions $h_{\alpha\beta}$,
    say defined whenever $|z| < \delta$ for some small $\delta$, such that
    $G(z) - G(z)^{-1} = \sum_{|\alpha|+|\beta|=2} h_{\alpha\beta}(z) z^{\alpha} \conj{z}^{\beta}$.
Hence, if $|z| \lt \delta$, then $|G(z) - G(z)^{-1}| \le C' |z|^2$ where $C' \defas \sum_{|\alpha|+|\beta|=2} \left(\sup_{|z| \lt \delta} |h_{\alpha\beta}(z)|\right)$
Recall \eqref{eq: definition of G} for the definition of $G$, this says that, if $\diastasis(x,y) < \delta$, then
\[
\Modulus{
                \left[\frac{|p(\xy)|^2}{p(\xx)p(\yy)}\right]^{1/2}-\left[\frac{p(\xx)p(\yy)}{|p(\xy)|^2}\right]^{1/2}
            } \le  C' \diastasis(x,y)^2
\]
If $\diastasis(x,y) \ge \delta$, then
\[
\Modulus{
                \left[\frac{|p(\xy)|^2}{p(\xx)p(\yy)}\right]^{1/2}-\left[\frac{p(\xx)p(\yy)}{|p(\xy)|^2}\right]^{1/2}
            } \le  C'' \diastasis(x,y)^2
\]
where
\[
C'' \defas \delta^{-2 } \sup_{x,y\in \P^n} \Modulus{
                \left[\frac{|p(\xy)|^2}{p(\xx)p(\yy)}\right]^{1/2}-\left[\frac{p(\xx)p(\yy)}{|p(\xy)|^2}\right]^{1/2}
            }
\]
Hence we obtain the desired inequality bysetting $ \BoundForLocDi \defas \max \left\{
C', C''
            \right\}
$.
\end{proof}

Let $V$ denote the Lebesgue measure on $\C^n$.
Equip the unit sphere $S^{2n-1}$ of $\C^n$ with its Haar measure, namely the unique rotationally invariant Borel probability measure.
The integral of a Borel measurable function $f: \C^n\to \C$ can be transformed into polar coordinates (see p.6 in \cite{ABR00}):
\begin{equation} \label{eq: TransformToPolarCoordinates}
\int_{\C^n} \! f(z) \, \frac{n! \, \d V(z)}{\pi^n} = 2n \int_0^\infty r^{2n-1}\int_{S^{2n-1}} \! f(r\xi) \, \d\xi \d r
\end{equation}
If $g:\PosReals \to \C$ is Borel measurable, then \eqref{eq: TransformToPolarCoordinates} simplifies to
\begin{equation} \label{eq: TransformRadiallySymmetricFunctionToPolarCoordinates}
\int_{\C^n} \! g(|z|) \, \frac{n! \, \d V(z)}{\pi^n} = 2n \int_0^\infty r^{2n-1} g(r) \, \d r
\end{equation}

\begin{lemma} If a function $R: \PosInt \to \PosReals$ satisfies $\lim_{m\approaches \infty} R(m) =0$,
                            then the following asymptotics hold uniformly for $x\in \P^n$:
\begin{align}
\int_{\P^n} \!\left[\frac{|r(\xy)|^2}{r(\xx)r(\yy)}\right]^m \,\Omega(y) &= \frac{n!}{m^n} + \bigO\left(\frac{1}{m^{n+1}}\right) \label{eq: TotalIntegralOfRM}\\
\int_{\{y:\, \diastasis(x,y) \ge R(m)\} } \!\left[\frac{|r(\xy)|^2}{r(\xx)r(\yy)}\right]^m \,\Omega(y)
&= \bigO\left(e^{-\frac{m}{2}R(m)^2}\right)\label{eq: OffDiagonalIntegralOfRM} \\
\int_{ \AffineBall{x}{R(m)} } \! \,\Omega(y)  &= \bigO(R(m)^{2n})\label{eq: NearDiagonalVolume}
\end{align}
\end{lemma}

\begin{proof}
Choose a canonical coordinate $z$ centered at $x$.
Recall that $\Omega$ is a normalization of the Fubini-Study volume form, hence there exists a constant $c \gt 0$ such that
\[
\Omega(z) = c\frac{n! \,\d V(z)}{\pi^n(1+|z|^2)^{n+1}}
\]
Hence
\begin{align}
\int_{\P^n} \left[\frac{|r(\xy)|^2}{r(\xx)r(\yy)}\right]^m\, \Omega(y) &= \int_{\C^n} \frac{1}{(1+|z|^2)^m} \, c\frac{n!\, \d V (z)}{\pi^n(1+|z|^2)^{n+1}} \notag \\
&= c\int_{\C^n} \frac{1}{(1+|z|^2)^{m+n+1}} \frac{n!\, \d V(z)}{\pi^n} \label{eq: alpha}
\end{align}
By polar coordinate formula \eqref{eq: TransformRadiallySymmetricFunctionToPolarCoordinates}:
\begin{align}
\int_{\C^n} \frac{1}{(1+|z|^2)^{m+n+1}} \frac{n!\, \d V(z)}{\pi^n} &=  2n \int_0^\infty r^{2n-1} \frac{1}{(1+r^2)^{m+n+1}}  \d r \notag \\
                                                                                                &=    \frac{n!}{(m+n)(m+n-1)\cdots(m+1)} \label{eq: beta}
\end{align}
Combine \eqref{eq: alpha} and \eqref{eq: beta} to obtain
\[
\int_{\P^n} \left[\frac{|r(\xy)|^2}{r(\xx)r(\yy)}\right]^m\, \Omega(y) = c\frac{n!}{(m+n)(m+n-1)\cdots(m+1)}
\]
In particular, when $m=0$, this becomes $\int_{\P^n} \Omega = c$. Hence $c= 1$, by our normalization of $\Omega$. This proves \eqref{eq: TotalIntegralOfRM}:
\begin{align*}
\int_{\P^n} \left[\frac{|r(\xy)|^2}{r(\xx)r(\yy)}\right]^m\, \Omega(y) &= \frac{n!}{(m+n)(m+n-1)\cdots(m+1)} \\
&= \frac{n!}{m^n} + \bigO\left(\frac{1}{m^{n+1}}\right)
\end{align*}

Next we show \eqref{eq: OffDiagonalIntegralOfRM}.
Note that
\[
\left[\frac{|r(\xy)|^2}{r(\xx)r(\yy)}\right]^m = \frac{1}{(1+\diastasis(x,y)^2)^m}
\]
Hence
\begin{align*}
\int_{\{y:\, \diastasis(x,y) \ge R(m)\} } \!\left[\frac{|r(\xy)|^2}{r(\xx)r(\yy)}\right]^m \,\Omega(y)
&\le  \int_{\{y:\, \diastasis(x,y) \ge R(m)\} } \! \frac{1}{(1+\diastasis(x,y)^2)^m} \,\Omega(y) \\
&\le  \frac{1}{(1+R(m)^2)^m} \int_{\{y:\, \diastasis(x,y) \ge R(m)\} } \!  \,\Omega(y) \\
&\le \frac{1}{(1+R(m)^2)^m} \int_{\P^n}  \! \,\Omega(y) \\
&= \frac{1}{(1+R(m)^2)^m}
\end{align*}
The inequality $1+\epsilon \ge e^{\epsilon / 2}$ holds for small $\epsilon \gt 0$.
By the assumption $\lim_{m\approaches \infty} R(m) =0$, hence
\[
\frac{1}{1+R(m)^2} \le \frac{1}{e^{R(m)^2/2 }}
\]
Hence \eqref{eq: OffDiagonalIntegralOfRM} follows.

Finally we prove \eqref{eq: NearDiagonalVolume}.
In the canonical coordinate $z$, the volume form $\Omega$ has an upper bound
\begin{align*}
 \Omega
&=  \frac{n!\, \d V (z)}{\pi^n(1+|z|^2)^{n+1}} \\
&\le \frac{n!\, \d V (z)}{\pi^n}
\end{align*}
Integrating, the polar coordinate formula \eqref{eq: TransformRadiallySymmetricFunctionToPolarCoordinates} gives
$\int_{ \AffineBall{x}{R(m)} } \! \,\Omega(y) \le 2n\int_0^{R(m)} \! r^{2n-1} \, \d r = R(m)^{2n} $.
Hence \eqref{eq: NearDiagonalVolume} follows.
\end{proof}

\begin{lemma} \label{lem: ApproximateError}
Let $R_0\ge 0$.
If $q: \P^n\times \P^n \to \PosReals$ and $g: \P^n\to\PosReals$ are continuous functions, then
\begin{multline} \label{eq: inequality in ApproximateError}
\int_{\P^n} \int_{\AffineBall{x}{R_0}} \! q(x,y) g(x) g(y) \, \Omega(y) \Omega(x) \\
\le \sqrt{\sup_{x\in \P^n}\int_{\AffineBall{x}{R_0}} \! q(x,y)^2 \,\Omega(y)}  \sqrt{ \int_{\AffineBall{\bullet}{R_0}} \Omega} \int_{\P^n} \! g^2 \,\Omega
\end{multline}
\end{lemma}
    For any two points $x$ and $x'$ on $\P^n$,
        the integrals $\int_{\AffineBall{x}{R_0}} \Omega$ and $\int_{\AffineBall{x'}{R_0}} \Omega$ are equal.
Let $\int_{\AffineBall{\bullet}{R_0}} \Omega$ denote this particular value.

\begin{proof}
For convenience, we suppress the integrand of \eqref{eq: inequality in ApproximateError} in the notation. That is to say, when a single or double integral appears without integrand, the reader understands that we refer respectively to the inner or double integral on the lefthand side of \eqref{eq: inequality in ApproximateError}.

Suppose $x$ is a point on $\P^n$.
By the Schwarz inequality,
\begin{align*}
&\phantom{=}\int_{\AffineBall{x}{R_0}} \! q(x,y) g(x) g(y) \, \Omega(y) \\
&= g(x)\int_{\AffineBall{x}{R_0}} \! q(x,y)  g(y) \, \Omega(y) \\
&\le g(x) \sqrt{\int_{\AffineBall{x}{R_0}} q(x,y)^2 \, \Omega(y)} \sqrt{\int_{\AffineBall{x}{R_0}} g(y)^2 \, \Omega(y) } \\
&\le g(x) \sqrt{\sup_{x\in \P^n}\int_{\AffineBall{x}{R_0}} \! q(x,y)^2 \,\Omega(y)}\sqrt{\int_{\AffineBall{x}{R_0}} g(y)^2 \, \Omega(y) }
\end{align*}
Integrating with respect to $x$,
\begin{align*}
&\phantom{=}\int_{\P^n} \int_{\AffineBall{x}{R_0}} \! q(x,y) g(x) g(y) \, \Omega(y) \Omega(x) \\
&\le \int_{\P^n} \! g(x) \sqrt{\sup_{x\in \P^n}\int_{\AffineBall{x}{R_0}} \! q(x,y)^2 \,\Omega(y)}\sqrt{\int_{\AffineBall{x}{R_0}} g(y)^2 \, \Omega(y) } \, \Omega(x) \\
&= \sqrt{\sup_{x\in \P^n}\int_{\AffineBall{x}{R_0}} \! q(x,y)^2 \,\Omega(y)} \int_{\P^n} \! g(x) \sqrt{\int_{\AffineBall{x}{R_0}} g(y)^2 \, \Omega(y) } \, \Omega(x)
\end{align*}
Apply the Schwarz inequality again:
\begin{multline} \label{eq: temp estimate in ApproximateError}
\int_{\P^n} \int_{\AffineBall{x}{R_0}} \le \sqrt{\sup_{x\in \P^n}\int_{\AffineBall{x}{R_0}} \! q(x,y)^2 \,\Omega(y)}\\
 \cdot \sqrt{\int_{\P^n} \! g(x)^2\,\Omega(x)} \sqrt{\int_{\P^n} \! \int_{\AffineBall{x}{R_0}} g(y)^2 \, \Omega(y) \Omega(x)}
\end{multline}

By the Fubini theorem, we compute using the Iverson bracket notation,
\begin{align*}
&\phantom{=}\int_{\P^n} \! \int_{\AffineBall{x}{R_0}} g(y)^2 \, \Omega(y) \Omega(x) \\
&= \int_{\P^n}\int_{\P^n} \! [\diastasis(x,y) < R_0] g(y)^2 \, \Omega(y) \Omega(x) \\
&= \int_{\P^n} \! g(y)^2 \int_{\P^n} \! [y\in \AffineBall{x}{R_0}] \, \Omega(x) \Omega(y) \\
&= \int_{\P^n} \! g(y)^2 \int_{\{x:\, \diastasis(x,y) < R_0\}} \! \, \Omega(x) \, \Omega(y) \\
&= \int_{\AffineBall{\bullet}{R_0}} \Omega \int_{\P^n} \! g(y)^2 \,\Omega(y)
\end{align*}
Hence, by \eqref{eq: temp estimate in ApproximateError},
\begin{align*}
&\phantom{=}\int_{\P^n} \! \int_{\AffineBall{x}{R_0}} g(y)^2 \, \Omega(y) \Omega(x) \\
&\le \sqrt{\sup_{x\in \P^n}\int_{\AffineBall{x}{R_0}} \! q(x,y)^2 \,\Omega(y)} \sqrt{\int_{\P^n} \! g(x)^2\,\Omega(x)}
                                \sqrt{\int_{\AffineBall{x}{R_0}} \Omega \int_{\P^n} \! g(y)^2 \,\Omega(y)} \\
&= \sqrt{\sup_{x\in \P^n}\int_{\AffineBall{x}{R_0}} \! q(x,y)^2 \,\Omega(y)}  \sqrt{ \int_{\AffineBall{x}{R_0}} \Omega} \int_{\P^n} \! g^2 \,\Omega\end{align*}
\end{proof}

Our normalization of $\Omega$ implies that $\int_{\AffineBall{x}{R_0}} \Omega \le \int_{\P^n} \Omega =1$.
Hence the above lemma has the following weaker form.

\begin{corollary} \label{cor: ApproximateWeaklyError}
Under the same conditions as in the above lemma, the following inequality holds:
\begin{multline*}
\int_{\P^n} \int_{\AffineBall{x}{R_0}} \! q(x,y) g(x) g(y) \, \Omega(y) \Omega(x) \\
\le \sqrt{\sup_{x\in \P^n}\int_{\AffineBall{x}{R_0}} \! q(x,y)^2 \,\Omega(y)}  \int_{\P^n} \! g^2 \,\Omega
\end{multline*}
\end{corollary}

\begin{lemma} \label{lem: ComputeRotationallyInvariantIntegrand}
Let $R_0 \in [0, \infty]$ and let $x$ be a point of $\P^n$.
If $h : [0,R_0)\to \C$ is a continuous function and $f$ is holomorphic on $\AffineBall{x}{R_0}$,
    then
\[
\int_{\AffineBall{x}{R_0}} \! h(\diastasis(x,y)) f(y) \, \Omega(y) = f(x) \int_{\AffineBall{x}{R_0}} \! h(\diastasis(x,y))\, \Omega(y)
\]
\end{lemma}

\begin{proof}
Define $\expr{J}$ as the difference between the two sides of the required identity.
Then $\expr{J} = \int_{\AffineBall{x}{R_0}} \! h(\diastasis(x,y)) \{f(y)-f(x)\} \, \Omega(y)$.
We wish to show that $\expr{J} = 0$.

Choose a canonical coordinate $z$ centered at $x$.
Write $\EuBall{R_0}$ for the Euclidean ball in $\C^n$ centered at the origin of radius $R_0$.
    In terms of this coordinate $z$,
    \begin{align*}
    \expr{J} &= \int_{\EuBall{R_0}} \! \LocExpOfQ(|z|) \{f(z) - f(0)\} \, \frac{n! \,\d V(z)}{\pi^n(1+|z|^2)^{n+1}} \\
    &= \int_{\EuBall{R_0}} \! \frac{ \LocExpOfQ(|z|) \{f(z) - f(0)\}}{(1+|z|^2)^{n+1}} \, \frac{n! \,\d V(z)}{\pi^n}
    \end{align*}
    Transform this integral to polar coordinates using \eqref{eq: TransformToPolarCoordinates}:
    \begin{align*}
    \expr{J} &= 2n \int_0^{R_0}  \! r^{2n-1} \int_{S^{2n-1}} \! \frac{ \LocExpOfQ(r) \{f(r\xi) -f(0)\}}{(1+r^2)^{n+1}} \,  \d\mu(\xi) \d r \\
        &= 2n \int_0^{R_0}  \! \frac{r^{2n-1} \LocExpOfQ(r)}{(1+r^2)^{n+1}}  \int_{S^{2n-1}} \!  \{f(r\xi) -f(0)\} \, \d\mu(\xi) \d r
    \end{align*}

    A holomorphic function $f$ is harmonic.
    By the mean value property of harmonic functions (see 1.4 of \cite{ABR00}), we have $\int_{S^{2n-1}} \!  \{f(r\xi) -f(0)\} \, \d\mu(\xi) = 0$.
    Thus $\expr{J} = 0$, which completes the proof.
\end{proof}

\section{Proof of main theorem} \label{sec: ProofOfAsymptotic}

Let $m$ be a nonnegative integer and $s \in H^0(\P^n, \twistSheaf{m +d})$.
    Suppose $R:\PosInt \to \PosReals$ is a function with $\lim_{m\approaches \infty} R(m) = 0$.
        We will choose a particular $R$ later.
By our definition \eqref{eq: DefineAssociatedIntegralOperator} of $\K_{r^mp}$,
\begin{align}
\K_{r^mp}(s,s)
&= \int_{\P^n} \int_{\P^n} \! \frac{(r^mp)(\xy)\conj{s(x)}s(y)}{(r^mp)(\xx)(r^mp)(\yy)} \, \Omega(y)\Omega(x) \label{eq: TheDoubleIntegral} \\
&= \expr{A} + \expr{B} + \expr{C},
\end{align}
where
\begin{align*}
\expr{A} &\defas \int_{\P^n} \int_{\AffineBall{x}{R(m)}} \!
                    \left[\frac{|r(\xy)|^2}{r(\xx)r(\yy)}\right]^m \frac{\conj{s(x)}s(y)}{(r^mp)(\yx)} \, \Omega(y)\Omega(x) \\
\expr{B} &\defas \int_{\P^n} \int_{\AffineBall{x}{R(m)}} \! \left[\frac{|r(\xy)|^2}{r(\xx)r(\yy)}\right]^m
                    \left[\frac{|p(\xy)|^2}{p(\xx)p(\yy)}-1\right]\frac{\conj{s(x)}s(y)}{(r^mp)(\yx)} \, \Omega(y)\Omega(x) \\
\expr{C} &\defas \int_{\P^n} \int_{\{y:\, \diastasis(x,y) \ge R(m)\}} \! \frac{(r^mp)(\xy)\conj{s(x)}s(y)}{(r^mp)(\xx)(r^mp)(\yy)} \, \Omega(y)\Omega(x)
\end{align*}
Term $\expr{A}$ will be dominant for our eventual choice of $R$.

First we compute $\expr{A}$.
    The zero section of the polarization of $p$ lies off the diagonal of $\P^n \times \P^n$.
    Hence for sufficiently large $m$, we have $R(m)$ small, so that for $\diastasis(x,y) < R(m)$,
            the expression $\frac{\conj{s(x)}s(y)}{(r^mp)(\yx)}$ is well-defined and holomorphic in $y$.
Note that
\[
\left[\frac{|r(\xy)|^2}{r(\xx)r(\yy)}\right]^m  = \frac{1}{(1 + \diastasis(x,y))^m}
\]
Hence, we may use Lemma \ref{lem: ComputeRotationallyInvariantIntegrand}, which gives for each $x$
\begin{multline*}
\int_{\AffineBall{x}{R(m)}} \!
                    \left[\frac{|r(\xy)|^2}{r(\xx)r(\yy)}\right]^m \frac{\conj{s(x)}s(y)}{(r^mp)(\yx)} \, \Omega(y) \\
= \frac{\conj{s(x)}s(x)}{r^mp(\xx)}
        \int_{\AffineBall{x}{R(m)}} \!\left[\frac{|r(\xy)|^2}{r(\xx)r(\yy)}\right]^m \,\Omega(y)
\end{multline*}
Integrating with respect to $x$,
\[
\expr{A}
    = \int_{\P^n} \! \frac{\conj{s(x)}s(x)}{r^mp(\xx)}
        \int_{\AffineBall{x}{R(m)}} \!\left[\frac{|r(\xy)|^2}{r(\xx)r(\yy)}\right]^m \,\Omega(y) \Omega(x) \\
\]
Taking the difference of \eqref{eq: TotalIntegralOfRM} and \eqref{eq: OffDiagonalIntegralOfRM},
\[
\int_{\AffineBall{x}{R(m)}} \!\left[\frac{|r(\xy)|^2}{r(\xx)r(\yy)}\right]^m \,\Omega(y)
    = \frac{n!}{m^n}+\bigO\left(\frac{1}{m^{n+1}} + \frac{1}{e^{\frac{m}{2}R(m)^2}}\right)
\]
Hence
\begin{align}
\expr{A} &= \left[\frac{n!}{m^n}+\bigO\left(\frac{1}{m^{n+1}} + \frac{1}{e^{\frac{m}{2}R(m)^2}}\right)\right]\int_{\P^n} \! \frac{\conj{s(x)}s(x)}{(r^mp)(\xx)} \, \Omega(x) \notag \\
&= \left[\frac{n!}{m^n}+\bigO\left(\frac{1}{m^{n+1}} + \frac{1}{e^{\frac{m}{2}R(m)^2}}\right)\right]\norm{s}^2 \label{eq: AsymptoticsOfA}
\end{align}

Next, we estimate $\expr{B}$.
By Lemma \ref{lem: BoundLocallyDiastasisOfP},
The modulus of its integrand is
\begin{align*}
& \Modulus{ \left[\frac{|r(\xy)|^2}{r(\xx)r(\yy)}\right]^m
                    \left[\frac{|p(\xy)|^2}{p(\xx)p(\yy)}-1\right]\frac{\conj{s(x)}s(y)}{(r^mp)(\yx)}
                  }  \\
&=  \left[\frac{|r(\xy)|^2}{r(\xx)r(\yy)}\right]^{m/2} \Modulus{\left[\frac{|p(\xy)|^2}{p(\xx)p(\yy)}\right]^{1/2}-\left[\frac{p(\xx)p(\yy)}{|p(\xy)|^2}\right]^{1/2}} \\
&\qquad\cdot\frac{|s(x)|}{(r^mp)^{1/2}(\xx)} \frac{|s(y)|}{(r^mp)^{1/2}(\yy)} \\
&\le  \BoundForLocDi \diastasis(x,y)^2\left[\frac{|r(\xy)|^2}{r(\xx)r(\yy)}\right]^{m/2} \frac{|s(x)|}{(r^mp)^{1/2}(\xx)} \frac{|s(y)|}{(r^mp)^{1/2}(\yy)}
\end{align*}
Hence
\begin{align*}
|\expr{B}| &\le \int_{\P^n} \int_{\AffineBall{x}{R(m)}} \! \Modulus{ \left[\frac{|r(\xy)|^2}{r(\xx)r(\yy)}\right]^m
                    \left[\frac{|p(\xy)|^2}{p(\xx)p(\yy)}-1\right]\frac{\conj{s(x)}s(y)}{(r^mp)(\yx)}
                  } \, \Omega(y)\Omega(x) \\
&= \BoundForLocDi R(m)^2 \int_{\P^n} \int_{\AffineBall{x}{R(m)}} \! \left[\frac{|r(\xy)|^2}{r(\xx)r(\yy)}\right]^{m/2} \frac{|s(x)|}{(r^mp)^{1/2}(\xx)} \frac{|s(y)|}{(r^mp)^{1/2}(\yy)} \, \Omega(y)\Omega(x)
\end{align*}
By Lemma \ref{lem: ApproximateError}, this becomes
\begin{align*}
|\expr{B}| &\le \BoundForLocDi R(m)^2 \sqrt{\int_{\AffineBall{x}{R(m)}} \! \left[\frac{|r(\xy)|^2}{r(\xx)r(\yy)}\right]^m \,\Omega(y)} \sqrt{ \int_{\AffineBall{\bullet}{R_0}} \Omega} \int_{\P^n} \! \frac{|s|^2}{r^mp} \, \Omega\\
&\le  \BoundForLocDi  R(m)^2 \sqrt{\int_{\P^n} \! \left[\frac{|r(\xy)|^2}{r(\xx)r(\yy)}\right]^m \,\Omega(y)}
        \sqrt{ \int_{\AffineBall{\bullet}{R_0}} \Omega} \norm{s}^2
\end{align*}
Hence, by the asymptotics \eqref{eq: TotalIntegralOfRM} and \eqref{eq: NearDiagonalVolume},
\begin{align}
\expr{B} &= \bigO\left(R(m)^2\right) \sqrt{\frac{n!}{m^n}} \sqrt{ \bigO\left(R(m)^{2n} \right)} \norm{s}^2  \notag\\
&= \bigO \left(\frac{R(m)^{n+2}}{m^{n/2}}\right) \norm{s}^2 \label{eq: AsymptoticsOfB}
\end{align}

Finally, we estimate $\expr{C}$.
Let
\[
M_p \defas \sup_{x,y\in\P^n} \left[\frac{|p(\xy)|^2}{p(\xx)p(\yy)}\right]^{1/2}
\]
By the compactness of $\P^n$, this positive constant $M_p$ is finite.
Hence
\begin{align*}
& \frac{|(r^mp)(\xy)||s(x)||s(y)|}{(r^mp)(\xx)(r^mp)(\yy)}  \\
&=  \left[\frac{|r(\xy)|^2}{r(\xx)r(\yy)}\right]^{m/2} \left[\frac{|p(\xy)|^2}{p(\xx)p(\yy)}\right]^{1/2} \frac{|s(x)|}{(r^mp)^{1/2}(\xx)} \frac{|s(y)|}{(r^mp)^{1/2}(\yy)} \\
&\le \BoundForGloDi \left[\frac{|r(\xy)|^2}{r(\xx)r(\yy)}\right]^{m/2} \frac{|s(x)|}{(r^mp)^{1/2}(\xx)} \frac{|s(y)|}{(r^mp)^{1/2}(\yy)}
\end{align*}
Hence
\[
|\expr{C}| \le \BoundForGloDi \int_{\P^n} \int_{\{y:\, \diastasis(x,y) \ge R(m)\}} \! \left[\frac{|r(\xy)|^2}{r(\xx)r(\yy)}\right]^{m/2}  \frac{|s(x)|}{(r^mp)^{1/2}(\xx)} \frac{|s(y)|}{(r^mp)^{1/2}(\yy)} \, \Omega(y)\Omega(x)
\]
By Corollary \ref{cor: ApproximateWeaklyError}, this becomes
\begin{align*}
|\expr{C}| &\le \BoundForGloDi \sqrt{\int_{\{y:\, \diastasis(x,y) \ge R(m)\}} \! \left[\frac{|r(\xy)|^2}{r(\xx)r(\yy)}\right]^m \,\Omega(y)} \int_{\P^n} \frac{|s|^2}{r^mp} \, \Omega \\
&= \BoundForGloDi \sqrt{\int_{\{y:\, \diastasis(x,y) \ge R(m)\}} \! \left[\frac{|r(\xy)|^2}{r(\xx)r(\yy)}\right]^{m} \,\Omega(y)} \norm{s}^2
\end{align*}
Hence, by the asymptotic \eqref{eq: OffDiagonalIntegralOfRM},
\begin{align}
\expr{C}&= \bigO(1) \sqrt{\bigO\left(\frac{1}{e^{\frac{m}{2}R(m)^2}}\right)} \norm{s}^2  \notag \\
&= \bigO\left(\frac{1}{e^{\frac{m}{4}R(m)^2}}\right) \norm{s}^2 \label{eq: AsymptoticsOfC}
\end{align}

Since $\K_{r^mp}(s,s) = \expr{A} + \expr{B} + \expr{C}$,
    combining \eqref{eq: AsymptoticsOfA}, \eqref{eq: AsymptoticsOfB} and \eqref{eq: AsymptoticsOfC},
\begin{align*}
\K_{r^mp}(s,s) &= \left\{\frac{n!}{m^n}+\bigO\left(\frac{1}{m^{n+1}} + \frac{1}{e^{\frac{m}{2}R(m)^2}}+\frac{R(m)^{n+2}}{m^{n/2}}+\frac{1}{e^{\frac{m}{4}R(m)^2}}\right)\right\} \norm{s}^2 \\
&= \left\{\frac{n!}{m^n}+\bigO\left(\frac{1}{m^{n+1}} +\frac{R(m)^{n+2}}{m^{n/2}}+\frac{1}{e^{\frac{m}{4}R(m)^2}}\right)\right\} \norm{s}^2 \label{eq: FinalAsymptotic}
\end{align*}

To complete the proof, it suffices to find $R$ such that $\lim_{m\approaches \infty}R(m) = 0$ and
\begin{equation}  \label{eq: RequiredAsymptotic}
\frac{1}{m^{n+1}} +\frac{R(m)^{n+2}}{m^{n/2}}+\frac{1}{e^{\frac{m}{4}R(m)^2}} = \bigO\left(\frac{(\log m)^{n+2}}{m^{n+1}}\right)
\end{equation}
Indeed, such a function is given by
\[
R(m) = \frac{\log m}{\sqrt{m}}
\]
\qed

\section{Application to Quillen's positivstellensatz} \label{sec: ProveTheoremFromAsymptotic}

Let $n$ be a positive integer. Let $\C[\ZZ]$ denote the complex polynomial algebra on the indeterminates $Z_0,\ldots, Z_n,\conj{Z_0},\ldots, \conj{Z_n}$. A {\emph{multiindex}} $\alpha$ is a sequence $(\alpha_0,\ldots,\alpha_n)$ of $n+1$ nonnegative integers whose {\emph{length}} $|\alpha|$ is $\alpha_0+\cdots+\alpha_n$. Given a nonnegative integer $d$, a {\emph{bihomogeneous polynomial of bidegree $(d,d)$}} is a finite sum $\sum_{|\alpha|=|\beta|=d} a_{\alpha\beta} Z^\alpha \conj{Z}^\beta$, where each $a_{\alpha\beta}$ is a complex scalar and $Z^\alpha \conj{Z}^\beta\defas Z_0^{\alpha_0}\cdots Z_n^{\alpha_n} \conj{Z_0}^{\beta_0} \cdots \conj{Z_n}^{\beta_n}$. This polynomial is said to be {\emph{Hermitian}} if $\conj{a_{\alpha\beta}}=a_{\beta\alpha}$ for each $\alpha$ and $\beta$. A polynomial is said to be {\emph{holomorphic}} if only the indeterminates $Z_0,\ldots, Z_n$ occur. Given a holomorphic polynomial $s(Z)$, write $|s(Z)|^2\defas s(Z)\conj{s(Z)}$.

With these concepts, we can state Quillen's Positivstellensatz.

\begin{corollary} \label{cor: QuillenPositivstellensatz}
Let $p$ be a Hermitian bihomogeneous polynomial of bidegree $(d,d)$. If $p(\zz)\gt 0$ for each point $z\neq 0$ in $\C^{n+1}$, then for sufficiently large $m$, there exists a basis $\{s_\eta\}_{|\eta|=m+d}$ of the holomorphic polynomials of degree $m+d$ such that
\begin{equation} \label{eq: QuillenCertificate}
(|Z_0|^2+\cdots+|Z_n|^2)^mp(\ZZ)=\sum_{|\eta|=m+d} |s_\eta(Z)|^2
\end{equation}
\end{corollary}


\begin{proof}
Recall from the introduction that we chose a basis $(\Phi_0,\ldots,\Phi_n)$ of $H^0(\P^n,\twistSheaf{1})$.
The map $Z_0\mapsto \Phi_0,\ldots, Z_n\mapsto \Phi_n$ induces a graded $\C$-algebra isomorphism
\begin{equation} \label{eq: iso}
C[Z] \to \bigoplus_{k=0}^\infty H^0(\P^n,\twistSheaf{k})
\end{equation}
This isomorphism induces the given Hermitian bihomogeneous polynomial $p$ of bidegree $(d,d)$ with a globalizable metric on $\twistSheaf{d}$,
    which we will also denote as $p$ by abuse of notation.

Recall the inner product on $H^0(\P^n, \twistSheaf{m+d})$ defined by \eqref{eq: DefineInnerProductOnVectorSpaceOfHolomorphicPolynomials}.
Choose an orthonormal basis $(e_\gamma)$ of $H^0(\P^n, \twistSheaf{m+d})$.
In terms of this basis, write $r^mp =\sum_{\gamma,\delta} c_{\gamma\delta} e_{\gamma} \conj{e_\delta}$.
Since this polynomial $r^mp$ is Hermitian, its coefficients form a Hermitian matrix $(c_{\delta\gamma})$.
Diagonalizing, there exists a unitary matrix $P=(P_{\gamma\eta})$ and a real-valued diagonal matrix $D=\diag(\ldots,\lambda_\eta,\ldots)$
    such that $(c_{\delta\gamma})=PDP^*$.
    In particular, we have $c_{\gamma\delta}=\sum_{\eta} P_{\gamma\eta} \lambda_\eta \conj{P_{\delta\eta}}$.
Hence, setting $f_\eta \defas \sum_{\gamma} P_{\gamma\eta} e_\gamma$,
\begin{align}
r^mp &= \sum_{\eta,\gamma,\delta} P_{\gamma\eta} \lambda_\eta \conj{P_{\delta\eta}} e_{\gamma} \conj{e_\delta} \notag \\
&= \sum_{\eta} \lambda_\eta \sum_{\gamma} P_{\gamma\eta} e_\gamma \conj{\sum_{\delta} P_{\delta\eta} e_\delta} \notag \\
&=  \sum_{\eta} \lambda_\eta f_\eta \conj{f_\eta} \label{eq: WeCanDiagonalizeRMP}
\end{align}

We claim that $(f_\eta)$ is an orthonormal basis of $H^0(\P^n, \twistSheaf{m+d})$.
    Indeed, the basis $(e_\gamma)$ is chosen to be orthonormal, hence
        \begin{align*}
    (f_\eta,f_\theta) &= \sum_{\gamma,\delta} P_{\gamma\eta} \conj{P_{\delta\theta}} (e_\gamma,e_\delta) \\
                            &= \sum_{\gamma} P_{\gamma\eta} \conj{P_{\gamma\theta}}
    \end{align*}
            The columns of a unitary matrix are orthonormal under the standard inner product.
    The matrix $P$ is unitary, hence $\sum_{\gamma} P_{\gamma\eta} \conj{P_{\gamma\theta}} = [\eta = \theta]$, the Kronecker delta.
Therefore $(f_\eta,f_\theta) = [\eta=\theta]$, which proves the claim.

By \eqref{eq: DefineInnerProductOnVectorSpaceOfHolomorphicPolynomials} and \eqref{eq: DefineAssociatedIntegralOperator},
    the inner product $\K_{r^mp}(f_\eta,f_\eta)$ is given by a double integral:
\begin{equation}
\K_{r^mp}(f_\eta,f_\eta) = \int_{\P^n} \int_{\P^n} \! \frac{(r^mp)(\xy)\conj{f_\eta(x)}f_\eta(y)}{(r^mp)(\xx)(r^mp)(\yy)} \, \Omega(y)\Omega(x)
\end{equation}
By \eqref{eq: WeCanDiagonalizeRMP} and the orthonormality of $\{f_\eta\}$, this becomes
\begin{equation} \label{eq: contractAlongFEta}
\begin{aligned}[b]
\K_{r^mp}(f_\eta,f_\eta) &= \int_{\P^n} \int_{\P^n} \! \frac{\sum_\theta \lambda_\theta f_\theta(x)\conj{f_\theta(y)}\conj{f_\eta(x)}f_\eta(y)}{(r^mp)(\xx)(r^mp)(\yy)} \, \Omega(y)\Omega(x) \\
&=\sum_{\theta} \lambda_\theta \int_{\P^n} \! \frac{f_\theta(x) \conj{f_\eta(x)}}{(r^mp)(\xx)} \, \Omega(x) \int_{\P^n} \! \frac{f_\eta(y) \conj{f_\theta(y)}}{(r^mp)(\yy)} \, \Omega(y) \\
&= \sum_{\theta} \lambda_\theta (f_\theta,f_\eta) (f_\eta,f_\theta) \\
&= \lambda_\eta
\end{aligned}
\end{equation}

By the main theorem, for sufficiently large $m$ and each $f_\eta$,
\[
\K_{r^mp}(f_\eta,f_\eta) = \left\{\frac{n!}{m^n}+ \bigO\left(\frac{(\log m)^{n+2}}{m^{n+1}}\right)\right\} \norm{f_\eta}^2
\]
A global section that forms part of a basis is necessarily nonzero, hence $\norm{f_\eta}^2 \neq 0$.
The above asymptotic has leading coefficient $n! \gt 0$, hence $\K_{r^mp}(f_\eta,f_\eta)\gt 0$ for $m$ large.
    From \eqref{eq: contractAlongFEta}, we get $\lambda_\eta\gt 0$.
    Thus \eqref{eq: WeCanDiagonalizeRMP} can be rewritten as $r^mp= \sum_\eta \Modulus{\sqrt{\lambda_\eta} f_\eta}^2$
        where $(\sqrt{\lambda_\eta} f_\eta)$ is a basis of $H^0(\P^n, \twistSheaf{m+d})$.
Apply the $\C$-algebra isomorphism \eqref{eq: iso} between $\C[Z]$ and $\bigoplus_{k=0}^\infty H^0(\P^n,\twistSheaf{k})$ to complete the proof.
\end{proof}

\end{document}